\documentclass{amsart}

\usepackage{enumerate}
\usepackage{mathptmx,stmaryrd}

\usepackage{hyperref}
\usepackage[nobysame]{amsrefs}

\usepackage[all]{xy}
\SelectTips{cm}{}

\theoremstyle{plain}

\newtheorem{theorem}{Theorem}[section]
\newtheorem{lemma}[theorem]{Lemma}
\newtheorem{corollary}[theorem]{Corollary}

\theoremstyle{definition}
\newtheorem{definition}[theorem]{Definition}
\newtheorem*{ack}{Acknowledgements}

\theoremstyle{remark}
\newtheorem{remark}[theorem]{Remark}

\newcommand{\bbZ}{\mathbb{Z}}

\newcommand{\mcG}{\mathcal{G}}

\newcommand{\sfA}{\mathsf{A}}
\newcommand{\sfB}{\mathsf{B}}
\newcommand{\sfC}{\mathsf{C}}
\newcommand{\sfD}{\mathsf{D}}
\newcommand{\sfP}{\mathsf{P}}
\newcommand{\sfT}{\mathsf{T}}

\newcommand{\mrE}{\mathrm{E}}
\newcommand{\mrD}{\mathrm{D}}
\newcommand{\mrG}{\mathrm{G}}

\newcommand{\proj}{\operatorname{\mathsf{proj}}}
\newcommand{\cone}{\operatorname{\mathsf{cone}}}
\newcommand{\susp}{\mathsf{\Sigma}}

\newcommand{\dbcat}{\mathsf{D}^{\mathsf b}}

\newcommand{\lra}{\longrightarrow}
\newcommand{\xra}{\xrightarrow}

\newcommand{\add}{\operatorname{add}}
\renewcommand{\dim}{{\operatorname{dim}\,}}
\newcommand{\dia}{{\,\diamond\,}}
\newcommand{\End}{\operatorname{End}}
\newcommand{\Ext}{\operatorname{Ext}}
\newcommand{\kos}[2]{{#2/\!\!/#1}}
\newcommand{\modf}{{\operatorname{mod}\,}}
\newcommand{\cent}[1]{{#1}^{\!\mathsf c}}

\newcommand{\Hom}{\operatorname{Hom}}

\begin{document}

\title[Annihilation and decompositions]{Annihilation of cohomology and\\
decompositions of derived categories}

\author{Srikanth B. Iyengar}
\address{Department of Mathematics, University of Nebraska, Lincoln, NE 68588-0130, USA}
\email{s.b.iyengar@unl.edu}

\author{Ryo Takahashi}
\address{Graduate School of Mathematics, Nagoya University, Furocho, Chikusaku, Nagoya 464-8602, Japan}
\email{takahashi@math.nagoya-u.ac.jp}

\date{2nd August 2014}


\thanks{ SBI was partly supported by NSF grant DMS-1201889; RT was partly supported by JSPS Grant-in-Aid for Young Scientists (B) 22740008, JSPS Grant-in-Aid for Scientific Research (C) 25400038 and JSPS Postdoctoral Fellowships for Research Abroad.}

\begin{abstract}
It is proved that an element $r$ in the center of a coherent ring $\Lambda$ annihilates $\mathrm{Ext}^{n}_{\Lambda}(M,N)$, for some positive integer $n$ and all finitely presented $\Lambda$-modules $M$ and $N$, if and only if the bounded derived category of $\Lambda$ is an extension of the subcategory consisting of complexes annihilated by $r$ and those obtained as $n$-fold extensions of $\Lambda$. This has applications to finiteness of dimension of derived categories.
\end{abstract}

\maketitle

\section{Introduction}

Let $\Lambda$ be a right coherent ring, $\modf\Lambda$ the category of finitely presented right $\Lambda$-modules, and $\dbcat(\Lambda)$  its bounded derived category. The purpose of this note is to prove the result below that reveals a close link between the existence of uniform annihilators of Ext-modules,  as modules over the center $\cent\Lambda$ of $\Lambda$, and a  kind of decomposition of the derived category. In the statement, $\mcG$ is the class of morphisms in $\dbcat(\Lambda)$ that induce the zero map in cohomology; $r$ an element in $\cent\Lambda$; and
$\dbcat(\Lambda)_{r}$ consists of complexes $X$ with $r \Ext^{0}_{\Lambda}(X,X)=0$, whilst $\sfC \dia \sfD $ is the subcategory of complexes obtained as extensions of complexes in $\sfC$ and $\sfD$;  see \ref{def:diamond}.

\begin{theorem}
Fix a non-negative integer $n$ and an element $r$  in $\cent\Lambda$. The following conditions on $\dbcat(\Lambda)$  are equivalent.
\begin{enumerate}[{\quad\rm(1)}]
\item
$r\mcG^n=0$
\item
$\dbcat(\Lambda)=\dbcat(\Lambda)_r\dia \{\Lambda\}^{n\dia}$
\item
$\dbcat(\Lambda)=\{\Lambda\}^{n\dia}\dia \dbcat(\Lambda)_r$
\end{enumerate}
When they hold $r\Ext_{\Lambda}^n(\modf \Lambda,\modf \Lambda)=0$; conversely the latter condition gives $r^{3}\mcG^{2n}=0$.
\end{theorem}

This  result is a consequence of Theorem~\ref{thm:main} that applies to  abelian categories with enough projectives. In fact, the equivalence of conditions (1)--(3), and the proofs, carry over verbatim to generating projective classes in triangulated categories, in the sense of Christensen~\cite{Christensen98}; with Ext as in Section 4 of \emph{op.~cit.}, the entire statement carries over.

\medskip

Here is one application (see Corollary~\ref{cor:dim}) of the theorem above: If $r\in \cent\Lambda$ is a non-zerodivisor on $\Lambda$ and  satisfies $r\mcG^{n}=0$, then there is an inequality
\[
\dim \dbcat(\Lambda) \leq \dim \dbcat(\Lambda/r\Lambda) + n 
\]
concerning dimensions of the appropriate triangulated categories, in the sense of Rouquier \cite{Rouquier08a}. 
This inequality gives a way to deduce the finiteness of the dimension of the derived category of $\Lambda$ from that of the derived category of $\Lambda/r\Lambda$. The point is that the ring $\Lambda/r\Lambda$ is ``smaller'' than $\Lambda$; for example, the Krull dimension of $\cent{(\Lambda/r\Lambda)}$ is strictly smaller than that of $\cent\Lambda$. This approach is predicated on the existence of non-zero divisors that annihilate Ext-modules. For  results in this direction, see \cite{IyengarTakahashi14a}*{Section 7}.

\ack
We should like to thank the referee for suggestions concerning presentation.

\section{Decompositions}
We deduce the statement in the Introduction from Theorem~\ref{thm:main} below that concerns derived categories of abelian categories.

\begin{definition}
\label{def:diamond}
Let $\sfT$ be a triangulated category, and $\susp$ its suspension functor; soon we will focus on the derived category of an abelian category.

Let $\sfC$ be a subcategory (always assumed to be full) of $\sfT$. We write $\add(\sfC)$ for the smallest subcategory of $\sfT$ containing $\sfC$ and closed under finite direct sums, retracts, and shifts.  Given a subcategory $\sfD$ of $\sfT$, the subcategory consisting of objects $E$ that appear in exact triangles of the form
\[
C \to E \to D \to \susp C \quad\text{with $C\in\sfC$ and $D\in\sfD$},
\]
is denoted  $\sfC\ast\sfD$. It is convenient to introduce also the following notation:
\[
\sfC\dia\sfD := \add(\sfC\ast\sfD)\,.
\]
It is a consequence of the octahedral axiom that there  are equalities
\[
(\sfB\ast\sfC)\ast\sfD= \sfB\ast(\sfC\ast\sfD) \text{ and } (\sfB\dia\sfC)\dia\sfD= \sfB\dia(\sfC\dia\sfD)\,.
\]
In particular, we may well denote them $\sfB\ast\sfC\ast\sfD$ and $\sfB\dia\sfC\dia\sfD$, respectively.
\end{definition}

Throughout the rest of this section, $R$ will be a commutative ring.

\begin{definition}
An additive category $\sfA$ is said to be \emph{$R$-linear} if for each $A$ in $\sfA$ there are homomorphisms of rings
\[
\eta_{A}\colon R\to \End_{\sfA}(A)
\]
with the property that the action of $R$ on $\Hom_{\sfA}(A,B)$ induced by $\eta_{A}$ and $\eta_{B}$ coincide, for all $A,B$ in $\sfA$. Said otherwise, $\Hom_{\sfA}(A,B)$ is an $R$-module and this structure is compatible with compositions in $\sfA$. 
\end{definition}

Let $\sfA$  be an $R$-linear Abelian category. The category of complexes over $\sfA$ inherits an $R$-linear structure, as does the bounded derived category, $\dbcat(\sfA)$,  of $\sfA$. In either case, the action is compatible with the suspension, in that the morphisms $\susp(X\xra{r} X)$ and $\susp X \xra{r} \susp X$ coincide for all $r\in R$ and complexes $X$. What is used repeatedly in the sequel is that for any $r\in R$ and morphism $f\colon X\to Y$, in either category, there is an induced commutative square 
\[
\xymatrix{
X \ar@{->}[r]^{f} \ar@{->}[d]_{r} & Y \ar@{->}[d]^{r} \\
X \ar@{->}[r]^{f} & Y} 
\]

Henceforth, we assume that $\sfA$ has enough projective objects, and write $\proj \sfA$ for the corresponding subcategory.  For ease of notation, we  abbreviate
\begin{align*}
\sfT &:= \dbcat(\sfA) \\
\sfP_{n} & := \underset{\text{$n$ copies}}{\underbrace{\proj\sfA\, \dia\cdots \dia\, \proj\sfA }}\quad \text{for each $n\ge0$.}
\end{align*}
Recall that \emph{ghost} in $\sfT$ is a morphism $f\colon X\to Y$ such that 
\[
\Hom_{\sfT}(\susp^{n}P,f) = 0 \quad\text{for all $P$ in $\proj\sfA$ and $n\in\bbZ$.}
\]
In what follows, we write $\mcG$ for the class of ghosts; it is an ideal in $\sfT$.  For any integer $n$, the ideal $\mcG^n$ consists of morphisms that are $n$-fold compositions of ghosts.

\begin{remark}
\label{rem:ghosts}
For each non-negative integer $n$, one has
\[
\Hom_{\sfT}(P,g)=0\quad\text{for all $P\in\sfP_{n}$ and $g\in\mcG^n$}.
\]
This is the well-known Ghost Lemma; for a proof, see, for example, \cite{Kelly65}*{Theorem~3}.
\end{remark}

\begin{remark}
\label{rem:syzygy}
For each complex $X$ in $\sfT$ and integer $n\ge1$,  there is an exact triangle 
\[
P \xra{\ p\ } X \xra{\ q\ } Y \lra \susp P
\]
with $P$ in $\sfP_n$ and $q$ in $\mcG^{n}$; one can get this from, for instance, the construction of an Adams resolution of $X$; see \cite{Christensen98}*{Section 4}. When $X$ is in $\sfA$, such a triangle exists with $\susp^{-n}Y$ in $\sfA$.
\end{remark}

\begin{definition}
For $r\in R$, let $\sfT_r$ denote the subcategory of $\sfT$ consisting of complexes $X$ such that the multiplication morphism $X\xra{r}X$ is zero in $\sfT$; in other words, $r$ is in the kernel of the natural map $R\to \End_{\sfT}(X)$.
\end{definition}

\begin{remark}
\label{rem:ann}
Let $r,s$ be elements of $R$. In any exact triangle $X\to Y\to Z\to \susp X$ in $\sfT$, if $X\in\sfT_r$ and $Z\in\sfT_s$, then $Y\in\sfT_{rs}$ holds.

Indeed,  this is a well-known argument (analogous to one for the Ghost Lemma) contained in the commutative diagram below:
\[
\xymatrix{
 && Y\ar@{-->}[dl]\ar@{->}[d]^{s}  \ar@{->}[r]^{g} & Z \ar@{->}[d]^{s}  & \\
 &X\ar@{->}[r]^{f} \ar@{->}[d]^{r}  & Y\ar@{->}[d]^{r}\ar@{->}[r]^{g} &Z\ar@{->}[r]&\susp X \\
&X\ar@{->}[r]^{f} & Y&& 
}
\]
The squares in the diagram are commutative by the definition of the $R$-action on $\sfT$. The morphism $Y\to X$ exists because 
$gs = sg  =0$; the second equality holds as $Z$ is in $\sfT_{s}$. The morphism $Y\xra{rs}Y$ thus factors as 
$Y\to X\xra{r} X \xra{f} Y$ and hence is zero, for  $X$ is in $\sfT_{r}$. 
\end{remark}

In what follows, given a morphism $f\colon X\to Y$ of complexes over $\sfA$, its mapping cone is denoted $\cone(f)$; thus
\[
\cone(f)^{n} := Y^{n}\bigoplus X^{n+1}\quad\text{with differential } \begin{bmatrix} d^{Y} & f \\ 0 &- d^{X}  \end{bmatrix}
\]
The canonical exact sequence of complexes
\[
0\lra Y\lra \cone(f)\lra \susp X\lra 0
\]
gives rise to an exact triangle $X\xra{f}Y\to \cone(f) \to \susp X$ in $\sfT$.

\begin{remark}
\label{rem:koszul}
For $r\in R$ and complex $X$ over $\sfA$, set $\kos rX:=\cone(X\xra{r}X)$. Observe that $\kos rX$ is in $\sfT_{r}$ for the map
\[
\begin{bmatrix} 0 & 0 \\ 1 & 0 \end{bmatrix}\colon \kos rX \lra \kos rX
\]
defines a homotopy between multiplication by $r$ and the zero morphism.
\end{remark}

\begin{lemma}
\label{lem:symmetry}
For each subcategory $\sfC$ of\, $\sfT$ and element $r\in R$ there are inclusions
\[
\sfT_r\ast\sfC\subseteq\sfC\ast\sfT_{r^2}\quad\text{and}\quad \sfC\ast\sfT_r\subseteq\sfT_{r^2}\ast\sfC\,.
\]
\end{lemma}

\begin{proof}
We verify the first inclusion; the second one can be checked along the same lines.

Fix an $X$ in $\sfT_r\ast\sfC$. Thus, there exist $T\in\sfT_r$ and $C\in\sfC$ and an exact triangle in the top row of the following diagram:
\[
\xymatrix{
T\ar@{->}[r] & X \ar@{->}[r]^{f}  & C\ar@{->}[r]^{g} \ar@{-->}[dl]_{h}\ar@{->}[d]^{r}& \susp T \ar@{->}[d]^{r} \\
                   & X \ar@{->}[r]_{f}  & C\ar@{->}[r]_{g} & \susp T }
\]
The map $h$ exists because $gr = rg=0$, where the second equality holds because $T$ is in $\sfT_{r}$. By the octahedral axiom, the factorization $r=fh$ gives rise to an exact triangle 
\[
T\lra \cone(h) \lra \kos rC \lra
\]
It follows from Remarks~\ref{rem:ann} and \ref{rem:koszul} that $r^{2}$ annihilates $\cone(h)$. It remains to notice the exact triangle $C\lra X \lra \cone(h)\to \susp C$.
\end{proof}

\begin{definition}
For an element $r\in R$ and an integer $n\ge0$ we consider the following four conditions on the triangulated category $\sfT:=\dbcat(\sfA)$.
\begin{alignat*}{2}
&\mrD_{r,n}\quad \sfT =\sfT_r\dia \sfP_n\,, \quad\text{and}\quad 
&&\mrE_{r,n}\quad \ r\Ext_\sfA^n(\sfA,\sfA)=0\,,  \\
&\mrD'_{r,n}\quad \sfT =\sfP_n \dia \sfT_r\,, \quad\text{and}\quad 
&&\mrG_{r,n}\quad  r\mcG^n=0\,.
\end{alignat*}
\end{definition}

The statement from the introduction is a consequence of the following theorem. 

\begin{theorem}
\label{thm:main}
The following implications hold
\[
\mrD'_{r,n}\Longleftrightarrow \mrD_{r,n}  \Longleftrightarrow  \mrG_{r,n}\Longrightarrow \mrE_{r,n} \Longrightarrow \mrD_{r^{3},2n}
\]
\end{theorem}

\begin{proof}
($\mrD'_{r,n}\Rightarrow \mrG_{r,n}$): Fix $f\colon X\to Y$ be in $\mcG^{n}$, and $P\xra{p}X\xra{q} T\to \susp P$ the exact triangle provided by the hypothesis. Consider the commutative diagram below where the morphism $X\to P$ is induced by the fact the $qr = rq=0$,
since $T$ is in $\sfT_{r}$.
\[
\xymatrixrowsep{1.5pc}
\xymatrix{
 &X\ar@{-->}[dl] \ar@{->}[r]^{q} \ar@{->}[d]^{r} & T \ar@{->}[d]^{r} \\
P\ar@{->}[r]^{p}\ar@{->}[dr]_{0} &X\ar@{->}[r]^{q} \ar@{->}[d]^{f} & T  \\
&Y } 
\]
It remains to note that the composition  $fp=0$, by Remark~\ref{rem:ghosts}.
\medskip

($\mrD_{r,n}\Rightarrow \mrG_{r,n}$) can be verified by an argument analogous to the one above.

\medskip

($\mrG_{r,n} \Rightarrow \mrD'_{r,n}$) and ($\mrG_{r,n} \Rightarrow \mrD_{r,n}$): Fix $X$ in $\sfT$ and $P\xra{\ p\ }X\xra{\ q\ } Y\to\susp P$ the exact triangle from Remark~\ref{rem:syzygy}. By hypothesis $rq=0$, so the octahedral axiom applied to the composition $rq$ gives rises to an exact triangle
\[
\susp P \lra Y \bigoplus \susp X \lra \kos rY\to \susp^{2}P\,.
\]
It remains to recall that $\kos rY$ is in $\sfT_{r}$, by  Remark~\ref{rem:koszul}, so that property $\mrD'_{r,n}$ holds. Applying the octahedral axiom to the map $qr$, which is also zero, shows that $\mrD_{r,n}$ holds as well.

\medskip

($\mrG_{r,n}\Rightarrow \mrE_{r,n}$): This holds because any morphism $f\colon A\to \susp^{n}B$, with $A,B$ in $\sfA$ is in $\mcG^{n}$; see Remark~\ref{rem:ghosts}.

\medskip

($\mrE_{r,n} \Longrightarrow \mrD_{r^{3},2n}$): For a start observe that $\sfA\subseteq\sfT_r\dia\sfP_{n}$; this follows by an argument along the lines of the one for $\mrG_{r,n} \Rightarrow \mrD'_{r,n}$ above. For a complex $X$ over $\sfA$ let $Z^{*}(X)$ and $B^{*}(X)$ denote the cycles and boundaries of $X$, respectively. There are canonical exact triangles
\begin{gather*}
Z^{*}(X) \lra X \lra \susp B^{*}(X)\lra \susp Z^{*}(X) \\
B^{*}(X) \lra Z^{*}(X) \lra H^{*}(X)\lra \susp B^{*}(X)
\end{gather*}
As $Z^{*}(X)$ and $B^{*}(X)$ are in $\add(\sfA)$, one gets the first of the following chain of inclusions
\begin{align*}
\sfT & \subseteq \sfA \dia \sfA \\
       & \subseteq (\sfT_r\dia \sfP_n) \dia (\sfT_r\dia \sfP_n)   \\
       &  \subseteq \sfT_r \dia \sfT_{r^{2}}\dia \sfP_n \dia \sfP_n \\
       &\subseteq \sfT_{r^{3}}\dia \sfP_{2n}
\end{align*}
The third inclusion holds by the associativity of $\dia$ and Lemma~\ref{lem:symmetry}. The last one holds by Remark~\ref{rem:ann}, and the definition of the $\sfP_{n}$. This is the desired implication.
\end{proof}

\subsection*{Non-zerodivisors}
Let now $\Lambda$ be a right coherent ring and $r\in \cent\Lambda$ an non-unit element in the center of $\Lambda$. The homomorphism of rings $\Lambda\to \Lambda/r\Lambda$ then induces, by restriction of scalars, an exact functor of triangulated categories
\[
\dbcat(\Lambda/r\Lambda) \lra \dbcat(\Lambda)
\]
Evidently, its image lies in the subcategory $\dbcat(\Lambda)_r$.

\begin{lemma}
\label{lem:modr}
When $r$ is a non-zerodivisor on $\Lambda$, the  functor $\dbcat(\Lambda/r\Lambda)\to\dbcat(\Lambda)_r$ is dense up to direct summands.
\end{lemma}

\begin{proof}
Since $r$ is  a non-zerodivisor on $\Lambda$, the canonical map $\kos r\Lambda \to H^{0}(\kos r\Lambda) \cong \Lambda/r\Lambda$ is a quasi-isomorphism in $\dbcat(\Lambda)$. This gives rise to an exact triangle 
\[
\Lambda \xra{\ r\ } \Lambda \lra \Lambda/r\Lambda \lra \susp \Lambda\,.
\]
 For any $X\in\dbcat(\Lambda)_r$, applying $X \otimes_\Lambda^{\bf L}-$ yields an exact triangle 
\[
X \xra{\ r\ } X \lra  X \otimes_{\Lambda}^{\bf L} (\Lambda/r\Lambda) \lra \susp X\,.
\]
Since the first morphism in this triangle is zero, one gets an isomorphism
\[
X \otimes_{\Lambda}^{\bf L} (\Lambda/r\Lambda) \cong X\oplus\susp X\,.
\]
Note that $X \otimes_{\Lambda}^{\bf L} (\Lambda/r\Lambda)$ is in the image of the functor $\dbcat(\Lambda/r\Lambda)\to\dbcat(\Lambda)$.
\end{proof}

\subsection*{Dimension}
Recall that the \emph{dimension} of a triangulated category $\sfT$, denoted $\dim\sfT$, is the least non-negative integer $d$ for which there exists an object $G$ such that  $\{G\}^{(d+1)\dia}=\sfT$; see \cite{Rouquier08a}*{Definition 3.2}.

The result below justifies the inequality stated in the introduction. Recall that $\mcG$ denotes the class of ghosts in $\dbcat(\Lambda)$.

\begin{corollary}
\label{cor:dim}
Let $\Lambda$ be a right coherent ring. If $r\in\cent\Lambda$ is a non-zerodivisor on $\Lambda$ and satisfies $r\mcG^{n}=0$ for some non-negative integer $n$, then there is an inequality
\[
\dim \dbcat(\Lambda) \leq \dim \dbcat(\Lambda/r\Lambda) + n 
\]
\end{corollary}

\begin{proof}
Part of the hypothesis is that $\dbcat(\Lambda)$ satisfies condition $\mrG_{r,n}$, in notation of Theorem~\ref{thm:main}.  Keeping in mind Lemma~\ref{lem:modr} and that $\proj \Lambda=\add \Lambda$,  \emph{op.cit.}\ yields 
\[
\dbcat(\Lambda) = \dbcat(\Lambda/r\Lambda) \dia \{\Lambda\}^{n\dia}\,.
\]
We have identified $\dbcat(\Lambda/r\Lambda)$ with its image in $\dbcat(\Lambda)$. If for some complex $F$ and integer $d$ one has
$\dbcat(\Lambda/r\Lambda)=\{F\}^{(d+1)\dia}$, then the equality above yields
\[
 \dbcat(\Lambda) = \{F\bigoplus \Lambda\}^{(d + n+1)\dia}\,.
\]
This  implies the desired inequality.
\end{proof}

\end{document}